\documentclass{article}
\usepackage{tikz-cd}
\usepackage{nicematrix}
\usepackage{colortbl}
\usepackage{color}

\usepackage{fullpage}
\usepackage{authblk}
\usepackage{blkarray}
\usepackage{amsmath}
\usepackage{amsthm}
\usepackage{amssymb}
\usepackage{bm}
\usepackage{setspace}
\usepackage[numbers,sort&compress]{natbib}

\usepackage{amsfonts}
\usepackage{eucal}
\usepackage{latexsym}
\usepackage{mathdots}
\usepackage{mathrsfs}
\usepackage{enumitem}

\newcommand{\be}{\begin{equation}}
\newcommand{\ee}{\end{equation}}
\newcommand{\ba}{\begin{eqnarray}}
\newcommand{\ea}{\end{eqnarray}}
\newcommand{\nn}{\nonumber}

\newcommand{\bal}{\begin{align}}
\newcommand{\eal}{\end{align}}
\newcommand{\baln}{\begin{align*}}
\newcommand{\ealn}{\end{align*}}

\newcommand{\bi}{\begin{itemize}}
\newcommand{\ei}{\end{itemize}}
\newcommand{\bn}{\begin{enumerate}}
\newcommand{\en}{\end{enumerate}}
\newcommand{\bbm}{\begin{bmatrix}}
\newcommand{\ebm}{\end{bmatrix}}
\newcommand{\bpm}{\begin{pNiceMatrix}}
\newcommand{\epm}{\end{pNiceMatrix}}
\newcommand{\bsm}{\left ( \begin{smallmatrix}}
\newcommand{\esm}{\end{smallmatrix} \right) }

\newcommand{\mr}{\ensuremath{\mathrm}}
\newcommand{\scr}{\ensuremath{\mathscr}}

\newcommand{\mc}{\ensuremath{\mathcal}}

\newcommand{\ov}{\ensuremath{\overline}}
\newcommand{\sm}{\ensuremath{\setminus}}

\newcommand{\Om}{\ensuremath{\Omega}}

\newcommand{\La}{\ensuremath{\Lambda }}
\newcommand{\la}{\ensuremath{\lambda }}

\newcommand{\eps}{\ensuremath{\epsilon }}

\def\C{\mathbb{C}}
\def\R{\mathbb{R}}

\def\N{\mathbb{N}}

\newcommand{\cH}{\ensuremath{\mathcal{H}}}
\newcommand{\cJ}{\ensuremath{\mathcal{J} }}

\def\nbdom{\mathrm{Dom} \, }
\def\nbran{\mathrm{Ran} \, }
\def\nbker{\mathrm{Ker} \, }
\def\nbdim{\mathrm{dim} \, }

\newcommand{\Addresses}{{
  \bigskip
  \footnotesize

  R.~T.~W.~Martin, \textsc{Department of Mathematics, University of Manitoba}\par\nopagebreak
  \textit{E-mail address:} \texttt{Robert.Martin@umanitoba.ca}

\vspace{1cm}

}}

\newcommand{\ip}[2]{\ensuremath{\langle {#1} , {#2} \rangle}}

\newcommand{\ran}[1]{\ensuremath{\mathrm{Ran} \left( {#1} \right) }}

\newcommand{\inv}[1]{\ensuremath{ \left[ 1/{#1} \right] }}

\numberwithin{equation}{section}
\numberwithin{subsection}{section}

\newtheorem{thm}{Theorem}

\newtheorem{lemma}{Lemma}
\newtheorem{prop}{Proposition}

\newtheorem*{thm*}{Theorem}
\newtheorem{thmA}{Theorem}

\newtheorem{corA}{Corollary}

\theoremstyle{definition}
\newtheorem{defn}{Definition}
\newtheorem{remark}{Remark}

\title{On unitary equivalence to a self-adjoint or doubly--positive Hankel operator}

\author{Robert T.W. Martin\thanks{Supported by NSERC grant 2020-05683}}
\affil{\footnotesize University of Manitoba}

\date{}

\begin{document}

\NiceMatrixOptions{cell-space-top-limit=1pt}
\NiceMatrixOptions{cell-space-bottom-limit=1pt}

\maketitle

\begin{abstract}
Let $A$ be a bounded, injective and self-adjoint linear operator on a complex separable Hilbert space. We prove that there is a pure isometry, $V$, so that $AV>0$ and $A$ is Hankel with respect to $V$, \emph{i.e.} $V^*A = AV$, if and only if $A$ is not invertible. The isometry $V$ can be chosen to be isomorphic to $N \in \N \cup \{ + \infty \}$ copies of the unilateral shift if $A$ has spectral multiplicity at most $N$. We further show that the set of all isometries, $V$, so that $A$ is Hankel with respect to $V$, are in bijection with the set of all closed, symmetric restrictions of $A^{-1}$.
\end{abstract}

\section{Introduction}

Let $\cH$ be a separable complex Hilbert space and let $V$ be an isometry on $\cH$. We will say that a bounded operator, $A \in \scr{L} (\cH )$, is \emph{$V-$Hankel} or \emph{Hankel with respect to $V$} if 
$$ V^* A = A V, $$ where $\scr{L} (\cH)$ denotes the Banach space of bounded linear operators on $\cH$. A bounded operator, $A \in \scr{L} (\cH )$, will be said to obey the \emph{doubly--positive Hankel condition} with respect to $V$, or $A$ is $V-$\emph{Hankel and doubly--positive} if $A$ is positive semi-definite, $V-$Hankel and $B := AV = V^* A$ is also positive semi-definite \cite{Hankpos2,Hankpos}. An orthonormal set $\{ h_k \} _{k=1} ^N \subseteq \cH$, $N \in \N \cup \{ + \infty \}$, is said to be a \emph{cyclic orthonormal basis} for a self-adjoint operator, $A$, if
$$ \cH = \bigoplus _{k=1} ^N \cH _k, \quad \mbox{where} \quad \cH _k := \bigvee _{j=0} ^\infty A^j h_k, $$ and $\bigvee$ denotes closed linear span. That is, $\cH _k$ and $\cH _j$ are orthogonal subspaces for $j \neq k$ and each $\cH _j$ is $A-$invariant.  We will call the minimal such $N \in \N \cup \{ + \infty \}$ the \emph{spectral} or \emph{cyclic multiplicity} of $A$. That is, $A=A^*$ has cyclic multiplicity at most $N$ if and only if it can be written as the orthogonal sum of $N$ self-adjoint operators, $A = \bigoplus _{j=1} ^N A_j$ which each have a single cyclic vector. Any such $A_j$, \emph{i.e.} any self-adjoint operator which has a cyclic vector, is said to be \emph{multiplicity--free}. Equivalently, each $A_j$ has simple spectrum. An isometry, $V$, is said to be \emph{pure}, if it is unitarily equivalent to several copies of the unilateral shift $S = M_z$, acting on vector--valued Hardy space, $H^2 \otimes \C ^n$, $n \in \N \cup \{ + \infty \}$. Here, $n = \nbdim \nbran V ^\perp$, will be called the \emph{multiplicity} of $V$.  The notation $\nbran A$ denotes the range of $A$, $\nbdim \nbran A$ is the dimension of the range of $A$ and $\nbran A ^\perp := \cH \ominus \nbran A$ is the orthogonal complement of $\nbran A$.  Recall that by the Wold decomposition, any isometry on a separable Hilbert space decomposes as $V = V_0 \oplus U$ where $V_0$ is pure and $U$ is unitary.

The following theorem is our main result:

\begin{thmA} \label{main}
Let $A$ be a bounded, injective and self-adjoint operator on a separable Hilbert space $\cH$ with cyclic multiplicity at most $N \in \N \cup \{ + \infty \}$. Then, there is a pure isometry, $V$, of multiplicity $N$ so that $A$ is $V-$Hankel and $AV>0$ if and only if $A$ is not invertible. 
\end{thmA}
\setcounter{thmA}{2}
\setcounter{corA}{1}

Let $\nbdom T$ denote the \emph{domain} of a linear operator, $T$. We will say that a closed linear operator, $T$, is \emph{positive semi-definite} and write $T \geq 0$ if $\ip{Tx}{x} _{\cH} \geq 0$ for all $x \in \nbdom T$. If $T$ is also injective, we will write $T>0$ and say that $T$ is \emph{positive}. Under the additional assumption that $A>0$, Theorem \ref{main} recovers a recent result of Niemiec \cite[Theorem 1]{Hankpos} as a special case:
\begin{corA} \label{DposHank}
Let $A \in \scr{L} (\cH)$ be positive semi-definite and injective with cyclic multiplicity $N \in \N \cup \{ + \infty \}$. Then, for any $m \geq N$, $m \in \N \cup \{ + \infty \}$, there exists a pure isometry, $V$, of multiplicity $m$ so that $A$ is $V-$Hankel and doubly--positive if and only if $A$ is not invertible.
\end{corA}

\begin{remark} \label{symmetry}
If $A=A^* \in \scr{L} (\cH )$ is injective, the polar decomposition of $A$ is $J|A|$, where $|A| := \sqrt{A^* A} >0$ and $J=J^*$ is a \emph{symmetry}, \emph{i.e.} a self-adjoint unitary operator. Indeed, if $f(t) =t$ on $[-\| A \| , \| A \| ]$ then $f(t) = -\chi _{[-\| A \| , 0 ]} (t) |t| + \chi _{[0, \| A \| ] } (t) |t|$, where $\chi _\Om$ denotes the characteristic function of a Borel set, $\Om \subseteq \R$. By the functional calculus, 
$$ A =  \underbrace{\left( -\chi _{[-\| A \| , 0 )} (A ) + \chi _{(0, \| A \| ] } (A) \right)} _{=: J} \cdot |A|, $$ and $J = J^* = J^{-1}$ is the difference of two spectral projections. If $J$ is a symmetry then $J^2 = I$ and the spectrum of $J$ is contained in $\{ \pm 1 \}$. Thus, $\cH$ decomposes as $\cH = \cH _+ \oplus \cH _-$ where $J h _\pm = \pm h _\pm$ for $h_\pm \in \cH _\pm$.  
\end{remark}

Our next result relates the condition that $A=A^*$ is injective, $V-$Hankel and $AV>0$ to double--positivity of $|A| := \sqrt{A^* A}$ with respect to $W := JV$ where $A=J|A|$ is the polar decomposition of $A$. 

\begin{thmA} \label{Hdpos}
Suppose that $A=A^* \in \scr{L} (\cH )$ is injective and Hankel with respect to a pure isometry, $V$, of multiplicity $N$ and that $AV >0$. If $A = J |A|$ is the polar decomposition of $A$, then $|A| >0$ is $W-$Hankel and doubly--positive where $W=JV$. If the Wold decomposition of $W$ is $W = W_0 \oplus U$ on $\cH = \cH _0 \oplus \cH '$ with pure part $W_0$ and unitary part $U$, then $W_0$ has multiplicity $N$, both $\cH _0, \cH '$ are $|A|-$invariant and either $U = I _{\cH '}$ or $\cH ' = \{ 0 \}$. 
\end{thmA}

\begin{remark}
It does not seem possible to conclude that the isometry, $W$, in the statement of the above theorem is pure in general. If, for example, $S=M_z$ is the unilateral shift on $H^2$, define a symmetry, $J$, on $H^2$ by $J1=z$, $Jz =1$ and $Jz^k = z^k$ for all $k \geq 2$. Then, $W=JS$ obeys $W^* 1 = S^* J 1 = S^*  z = 1$, so that $W$ cannot be pure. In the proof of sufficiency in Theorem \ref{main}, given a self-adjoint, injective and non-invertible $A \in \scr{L} (\cH )$, we construct a pure isometry, $V$, so that $A$ is $V-$Hankel and $AV >0$ using a specific restriction of $A^{-1}$ and spectral theory. If $A =J |A|$ is the polar decomposition of $A$, where $J$ is a symmetry, it may be possible that the specific form of $W=JV$ implies that it is pure in this case. We leave this as an open problem for future research.
\end{remark}

Given a self-adjoint, non-invertible and bounded $A$ with trivial kernel, $\nbker A = \{ 0 \}$, we further show that the set of all isometries, $V$, so that $A$ is $V-$Hankel are in bijective correspondence with the set of all closed and symmetric restrictions of $A ^{-1}$:

\begin{thmA} \label{main2}
Let $A \in \scr{L} (\cH )$ be bounded, self-adjoint, injective and not invertible. Then $A$ is Hankel with respect to an isometry, $V$, with pure part, $V_0$ of multiplicity $N$, if and only if there is a closed, symmetric restriction, $\inv{A}_0$, of $A^{-1}$ with dense domain in $\cH$  and equal deficiency indices $(N, N)$,
$$ \nbdim \ran{ \inv{A} _0 \pm iI } ^\perp = N, $$ so that $\nbran \inv{A} _0 = \nbran V$. Namely, if $V$ is an isometry with pure part of multiplicity $N$, so that $A$ is $V-$Hankel, then $ \inv{A} _0 := A^{-1} | _{\nbran AV}$ is a closed, symmetric restriction of $A^{-1}$ with deficiency indices $(N,N)$, and if $\inv{A} _0$ is a closed, symmetric restriction of $A^{-1}$ with dense domain in $\cH$ and deficiency indices $(N,N)$, then the isometry, $V$, appearing in the polar decomposition of $\inv{A} _0$ has pure part of multiplicity $N$ and $A$ is $V-$Hankel. If $\inv{A} _0$ is simple and $N>0$ then $V$ is pure, and conversely, if $V$ is pure and if $V^* A = AV >0$, then $\inv{A} _0$ is simple. 
\end{thmA}
In the above statement, recall that the \emph{deficiency indices}, $(n_+, n_-)$ of a closed, symmetric operator, $T$, are defined as 
$$ n _\pm := \nbdim \ran{ T \pm iI } ^\perp \in \N \cup \{ 0 \} \cup \{ + \infty \}, $$ and further recall that a symmetric operator, $T$, is called \emph{simple} if it has no self-adjoint restriction, $T ' = T | _{\scr{D}}$, to a dense domain, $\scr{D}$, in a non-trivial closed subspace, $\cJ$, of $\cH$ so that $\cJ$ is invariant for $T$ in the sense that $T \scr{D} \subseteq \cJ$  \cite[Sections 78--81]{Glazman}, \cite[Chapter X.1]{RnS2}, \cite{vN}. Here, if $T_k : \nbdom T_k \subseteq \cJ _k \rightarrow \cH _k$, $k=1,2$, are two linear transformations with (not necessarily dense) domains $\nbdom T_k \subseteq \cJ _k$, where $\cJ _k , \cH _k$ are Hilbert spaces so that $\cJ _1 \subseteq \cJ _2$, we say that $T_1$ is a \emph{restriction} of $T_2$, or that $T_2$ is an \emph{extension} of $T_1$ and write $T_1 \subseteq T_2$ if $\nbdom T_1 \subseteq \nbdom T_2$ and $T_1 = T_2 | _{\nbdom T_1}$.

Corollary \ref{DposHank} is a special case of Theorem \ref{main} which subsumes Theorem 1 of \cite{Hankpos}. Theorem \ref{main} also has significant overlap with Theorem 1 and Theorem 2 of \cite{Hankinv}, which solve the inverse spectral problem for self-adjoint Hankel (and block Hankel) operators. That is, the main result of \cite{Hankinv} is \cite[Theorem 1]{Hankinv}, which gives necessary and sufficient conditions for a self-adjoint operator to be unitarily equivalent to a Hankel operator acting on one copy of Hardy space, $H^2$. Theorem 2 of \cite{Hankinv} then gives necessary and sufficient conditions for a self-adjoint operator, $A$, to be unitarily equivalent to a `block' Hankel operator acting on vector--valued $H^2$, \emph{i.e.} for $A$ to be $V-$Hankel with respect to some pure isometry, $V$. This latter result is a straightforward consequence of \cite[Theorem 1]{Hankinv} and spectral theory. In order to describe the results of \cite{Hankinv} and their relationship to ours, recall the direct integral form of the spectral theorem for a self-adjoint operator $A \in \scr{L} (\cH)$: Given $A=A^*$, $A$ is unitarily equivalent to multiplication by the independent variable acting on a direct integral of Hilbert spaces. That is, there is a positive, finite Borel measure, $\mu$, on $\R$, and a `direct integral of Hilbert spaces', 
$$ \cH _{\oplus _\mu}  = \int _{\R} ^\oplus \cH (t) \, \mu (dt ), $$ which is a Hilbert space consisting of measurable functions, $f$, with $f(t) \in \cH (t)$, $\cH (t)$ is a Hilbert space for each $t \in \R$, $\cH (t) \neq \{ 0 \} \ \mu - a.e.$ and the inner product is
$$ \ip{f}{g} := \int _{-\infty} ^\infty \ip{f(t)}{g(t)}_{\cH (t)} \mu (dt). $$ (We will not make use of this construction in this paper and so we refer the reader to \cite{Hankinv} and \cite{BS} for precise details.) The self-adjoint operator $A$ is then unitarily equivalent to $M_t$, $(M_t f) (s) = s f(s)$ acting on its maximal domain in $\cH _{\oplus _\mu}$.  The \emph{spectral multiplicity function}, $\nu _A$, of $A$ is then defined $\mu -a.e.$ as $\nu _A (t) = \nbdim \cH (t)$. Theorem 1 of \cite{Hankinv} then states that if $A=A^* \in \scr{L} (\cH )$ has such a spectral representation $M_t$ on $\cH _{\oplus _\mu}$, then $A$ is unitarily equivalent to a Hankel operator if and only if:
\bi
    \item[(a)] either $A$ is injective or $\nbdim \nbker A = +\infty$,
    \item[(b)] $A$ is not invertible, and
    \item[(c)] $| \nu _A (t) - \nu _A (-t) | \leq 2, \ \mu_{ac}-a.e.$ and $|\nu _A (t) - \nu _A (-t) | \leq 1, \ \mu_s -a.e.$
\ei
Our main result, Theorem \ref{main}, also provides a sufficient condition for a self-adjoint operator $A \in \scr{L} (\cH )$ to be unitarily equivalent to a Hankel operator, but this is a weaker result than that of Theorem 1 in \cite{Hankinv}, even in the case where $A$ in injective.  Namely, Theorem \ref{main} implies that a sufficent condition for $A \in \scr{L} (\cH )$ to be unitarily equivalent to a self-adjoint Hankel operator acting on one copy of $H^2$ is that $A =A^*$ be multiplicity--free, injective and not invertible. As described in, \emph{e.g.} \cite[Section 1]{Hankpos}, a bounded self-adjoint operator, $A$, can be decomposed as the direct sum of $N \in \N \cup \{ + \infty \}$ self-adjoint operators with simple spectra, \emph{i.e.} $A$ has cyclic multiplicity at most $N$, if and only if the essential supremum of its spectral multiplicity function does not exceed $N$. Hence any such $A$ satisfying the sufficiency condition from Theorem \ref{main} described above also satisfies the conditions (a)--(c) of \cite[Theorem 1]{Hankinv} listed above, but this sufficiency condition from Theorem \ref{main} is not necessary. However, while proof of necessity of the conditions (a)--(c) in \cite[Theorem 1]{Hankinv} is straightforward, proof of sufficiency is lengthy and involved, and proof of our weaker sufficiency condition in Theorem \ref{main} is relatively short. Moreover, Theorem \ref{main}, Corollary \ref{DposHank} and Theorem \ref{Hdpos} contain new information and results not implied by those of \cite{Hankinv,Hankpos} or other previous work, and our Theorem \ref{main2} is completely new.

Our proof of Theorem \ref{main} is, in our opinion, novel and uses a variety of interesting techniques including the theory of unbounded symmetric operators and their self-adjoint extensions, the theory of unbounded sesquilinear forms and a modified version of the `unbounded Douglas factorization lemma', Lemma \ref{modDFL}. In particular, our proof of Theorem \ref{main} uses entirely different methods than the dynamical systems approach of \cite{Hankinv}, or the methods of \cite{Hankpos}.

\section{Preliminaries}

\subsection{Necessity}
 
The proof of necessity in Theorem \ref{main} is elementary: 
 
\begin{proof}{ (Necessity in Theorem \ref{main})}
Suppose that $V$ is a pure isometry and that $A=A^*$ is $V-$Hankel and invertible so that $A^{-1}$ is bounded and self-adjoint. Then $V^* A = A V$ implies that 
$$ A^{-1} V^* = V A^{-1}. $$ Hence, for any $h \in \cH$ and $k \in \N$,
\ba \| A^{-1} h \| & = & \| V^k A^{-1} h \| \nn \\
& = & \| A^{-1} V^{*k} h \| \nn \\
& \leq &  \| A ^{-1} \| \, \| V^{*k} h \|  \rightarrow 0. \nn \ea This proves that $A^{-1} \equiv 0$, a contradiction.
\end{proof}
In the proof above we have employed the notation $V^{*k} := (V^*) ^k$.

\subsection{Spectral Multiplicity}

As in \cite{Hankinv} and \cite[Proposition 2]{Hankpos}, proof of sufficiency in Theorem \ref{main} can be reduced to the case where $A$ is multiplicity--free using spectral theory. 

\begin{lemma}
Let $A \in \scr{L} (\cH )$ be self-adjoint, injective and not invertible with cyclic multiplicity $N \in \N \cup \{ + \infty \}$. Then given any $M \geq N$, $M \in \N \cup \{ + \infty \}$, $A$ can be decomposed as $A = \oplus _{j=1} ^M A_j$ where each $A_j = A_j^*$ is also injective but not invertible. 
\end{lemma}
\begin{proof}
This is a straightforward application of spectral theory. Since we assume that $A$ has cyclic multiplcity $N$, 
$$ A = \bigoplus _{j=1} ^N B_j \quad \mbox{acting on} \quad \cH = \bigoplus _{j=1} ^N \cH _j $$ where each $B_j =B_j ^*$ is injective and multiplicity--free. Since $A$ is not invertible, at least one of the $B_j$ must be non-invertible and we assume, without loss in generality, that each $B_j$ for $1 \leq j \leq m$ is non-invertible for some fixed $m \in \{ 1 , \cdots , N \}$. Since each $B_j$ is invertible for $m+1 \leq j \leq N$, for each such $j$ there is a $\delta _j >0$ so that the spectrum, $\sigma (B _j )$, of $B_j$ is contained in $(-\infty , -\delta _j] \cup [ \delta _j , +\infty )$. The basic idea is to decompose $\cH _m$ and $B_m$ into the orthogonal direct sum of $M-m+1$ `pieces', 
$$ \cH _m = \bigoplus _{j=m} ^{M} \cJ _j \quad \mbox{and} \quad B_m = \bigoplus _{j=m} ^{M} C_j, $$ where each $C_j=C_j^*$ is self-adjoint, injective but not invertible, and so that for $m+1 \leq j \leq N$, the spectrum of $C_j$ is contained in $(-\delta _j, \delta _j )$. Then setting $A_j := B_j$ for $1\leq j \leq m-1$, $A_m := C_m$ acting on $\cJ _m$, $A_j := B_j \oplus C_j$ acting on $\cH _j \oplus \cJ _j$ for $m+1 \leq j \leq N$ and $A_j := C_j$ for $N+1 \leq j \leq M$ will yield the desired decomposition.   

Now consider the interval $(-1, 1)$. Since $B_m$ is injective but not invertible, $0$ is not an isolated point in the spectrum of $B_m$ and each of the open sets 
$(-1/n , 0 ) \cup (0, 1/n)$, for $n \in \N$, have non-empty intersection with the spectrum of $B_m$. It follows that we can construct an infinite sequence $(\la _k ) _{k=1} ^\infty$, of points in the intersection of $(-1, 1)$ with the spectrum, $\sigma (B_m)$ of $B_m$, so that $\la _k \neq 0$, $\la _k \neq \la _j$ for any $j \neq k$, $j, k \in \N$ and $\la _k \rightarrow 0$. Moreover, for each $k$, there is an $\eps _k >0$ so that the intervals $I_k := (\la _k -\eps _k , \la _k + \eps _k ) \subseteq (-1, 1)$ are disjoint, \emph{i.e.} $I_k \cap I_j = \emptyset$ for $j\neq k$. Divide the union of the these intervals, $\La := \cup _{j=1} ^\infty (\la _k - \eps _k , \la _k + \eps _k )$ into $M  -m +1$ sets, $\La _j$, $m \leq j \leq M$ with trivial pair-wise intersections, $\La _j \cap \La _k = \emptyset$ for $j \neq k$, $\La = \cup _{j=m} ^{M} \La _j$, so that $0$ is in the closure of each $\La _j$. Hence, by construction, $0$ will also belong to the closure of $\La _j \cap \sigma (B_m)$. One could choose, for example, $\La _m := \cup _{j=1} ^\infty (\la _{2j} - \eps _{2j}  , \la _{2j} + \eps _{2j}  )$, then choose $\La _{m+1} = \cup _{j=1} ^\infty (\la _{4j-1} - \eps _{4j -1}  , \la _{4j-1} + \eps _{4j-1} )$, and so on. For each $m+1 \leq j \leq N$, $B_j$ is invertible so that there is a $\delta _j >0$ obeying $\sigma (B_j) \subseteq \R \sm (-\delta _j,  \delta _j)$. We then define $\Om _j := \La _{j} \cap (-\delta _j /2 , \delta _j /2)$ for $m+1 \leq j \leq N$, and $\Om _j := \La _j$ for $j=m$ and all $N+1 \leq j \leq M$.  

Next, for any $m+1 \leq j \leq N$, we set
$$ C_j := B_m|_{\cJ _j}, \ \cJ _j := \nbran \chi _{\Om _j} (B_m ) \quad \mbox{and} \quad A_j := B_j \oplus C_j \quad \mbox{acting on} \quad \cH _j \oplus \cJ _j. $$  Here, $\chi _{\Om _j} $ denotes the characteristic function of the Borel set $\Om _j$ and the \emph{spectral projection}, $\chi _{\Om _j} (B_m )$, is a self-adjoint projection defined by the functional calculus for self-adjoint operators. The $\cJ _j$ are pairwise orthogonal since $\Om _j \cap \Om _k = \emptyset$ for $j\neq k$. Hence, $\sigma (B_j) \subseteq (-\infty , - \delta _j ] \cup [\delta _j , +\infty)$ while $\sigma (C_j) \subseteq [- \delta _j /2 , \delta _j /2]$, $0 \in \ov{\Om _j \cap \sigma (C_j )} \subseteq \sigma (C_j)$, and it follows that for $m+1 \leq j \leq N$ each $A_j = B_j \oplus C_j$ is self-adjoint, injective, not invertible and multiplicity--free. Indeed, if $x_j$ is a cyclic vector for each $B_j$, then it is readily verified that $x_j \oplus \chi_{\Om _j} (B_m) x_m$ is cyclic for $A_j$, $m+1 \leq j \leq N$. For $1 \leq j \leq m-1$ we then set $A_j:=B_j$ and for $N+1 \leq j \leq M$ we set $A_j := C_j := B_m | _{\cJ _j}$ acting on $\cJ _j := \nbran \chi _{\Om _j} (B_m)$. Each of these is, again, self-adjoint, injective, not invertible and multiplicity--free. Finally, if $\Om := \cup _{j=m+1} ^M \Om _j$, then 
$\La _m = \Om _m \subseteq \R \sm \Om$ so that if we set $A_m := C_m := B_m | _{\cJ _m}$ where $\cJ _m := \nbran \chi _{\R \sm \Om} (B_m)$, then $B_m = \oplus _{j=m} ^M C_j$ and  $$ A = \bigoplus _{j=1} ^M A_j, $$ is the desired orthogonal direct sum decomposition of $A$ so that each $A_j=A_j^*$ is bounded, multiplicity--free, injective and not invertible.
\end{proof}

It follows that to prove sufficiency in Theorem \ref{main}, it suffices to consider the case where $A$ is multiplicity--free. That is, if $A=A^* \in \scr{L} (\cH)$ satisfies the hypotheses of Theorem \ref{main}, is non-invertible and has cyclic multiplicity $N$, then by the above lemma, for any $M \geq N \in \N \cup \{ + \infty \}$ we can decompose $A = \bigoplus _{j=1} ^M A_j$ where each $A_j$ is bounded, self-adjoint, multiplicity--free, injective and not invertible. Hence, if sufficiency in Theorem \ref{main} can be established in the multiplicity--free case, then applying this to each direct summand, $A_j$, yields a pure isometry, $V_j$, of multiplicity $1$ with the desired properties (i) and (ii). Then $V := \bigoplus _{j=1} ^M V_j$ will be a pure isometry of multiplicity $M$ satisfying conditions (i) and (ii) of Theorem \ref{main}.

\subsection{A simple symmetric operator} \label{simsym}

Assuming $A$ is self--adjoint, multiplicity--free, injective and not invertible, consider $A^{-1}$. This is a self--adjoint, multiplicity--free, densely--defined (and unbounded) operator on $\nbdom A^{-1} = \nbran A$.  Let $P (\cdot ) : \mr{Bor} (\R ) \rightarrow \mr{Proj} (\cH )$ denote the \emph{projection--valued measure} or \emph{spectral measure} of $A^{-1}$, where $\mr{Bor} (\R)$ denotes the $\sigma-$algebra of Borel subsets of $\R$ and $\mr{Proj} (\cH )$ denotes the lattice of self--adjoint projection operators in $\scr{L} (\cH)$. Let $x \in \cH$ be a cyclic vector for $A$ and define the (scalar--valued) positive Borel measure:
$$ \mu (\Om ) := \int _{-\infty} ^\infty (1+t^2) \ip{x}{P(dt) x}_{\cH}; \quad \quad \Om \in \mr{Bor} (\R ). $$ Observe that 
$$ \mu ( \R ) = + \infty, $$ and that $A ^{-1}$ is unitarily equivalent to the self-adjoint multiplication operator $M_t$, $(M_t f) (t) = t f(t)$ acting on $L^2 (\mu ) := L^2 ( \R , \mu )$. Further note that, by construction, $\mu$ obeys the \emph{Herglotz condition},
$$ \int _{-\infty } ^\infty \frac{1}{1+t^2} \mu (dt) < + \infty. $$ 
Identify $A ^{-1}$ with $M_t$ acting on its maximal, dense domain in $L^2 (\mu )$, consider the linear manifold,
$$ \scr{D} _0 := \left\{ h \in \nbdom M_t \left| \ \int _{-\infty} ^\infty h(t) \mu (dt) =0  \right. \right\}, $$ 
and set $\inv{A} _0 := A ^{-1} | _{\scr{D} _0}$. By \cite[Proposition 3.5, Lemma 3.6]{AMR} and the discussion following their proofs, $\inv{A} _0 $ is a closed, injective and simple symmetric operator with deficiency indices $(1,1)$, and $\nbdom \inv{A} _0 = \scr{D} _0$ is dense in $\cH$.  We refer the reader to \cite[Sections 78--81]{Glazman} and \cite[Chapter X.1]{RnS2} for an introduction to the theory of unbounded symmetric operators and their self-adjoint extensions.

\subsection{Positive semi-definite quadratic forms}

A sesquilinear (or quadratic) form, $q : \nbdom q \times \nbdom q \rightarrow \C$, where the \emph{form domain} of $q$, $\nbdom q \subseteq \mc{H}$, is dense in a Hilbert space, $\mc{H}$, is \emph{positive semi-definite} if $q (h , h ) \geq 0$ for all $h \in \nbdom q$. (Technically, the domain of $q$ is $\nbdom q \times \nbdom q$ where $\nbdom q$ is the form domain of $q$.) Given such a positive semi-definite quadratic form, $q$, with dense form domain $\nbdom q \subseteq \cH$, we say that $q$ is \emph{closed} if its form domain, $\nbdom q$, is complete with respect to the norm induced by the inner product,
$$ q(x,y) + \ip{x}{y} _\cH. $$ That is, $q$ is closed if its form domain is a Hilbert space with respect to the above inner product. A positive semi-definite quadratic form is \emph{closeable} if it has a closed \emph{extension}. Here, as for linear operators, given two positive semi-definite quadratic forms $q, q'$ with dense form domains $\nbdom q \subseteq \cJ$ and $\nbdom q' \subseteq \cH$, where $\cJ$ is a closed subspace of $\cH$, we say that $q$ is a \emph{restriction} of $q'$, or that $q'$ is an \emph{extension} of $q$, and we write $q \subseteq q'$, if $\nbdom q \subseteq \nbdom q'$ and $q = q' | _{\nbdom q \times \nbdom q}$.  Any closeable, positive semi-definite quadratic form with dense form domain in $\cH$ has a minimal closed extension, $\ov{q}$, with dense form domain $\nbdom \ov{q} \subseteq \cH$. Standard references for the theory of unbounded sesquilinear forms include \cite[Chapter VIII.6]{RnS1} and \cite[Chapter 6]{Kato}.

A positive semi-definite quadratic form $q$, with dense form domain in a Hilbert space, $\cH$, is closed if and only if there is a unique self-adjoint and positive semi-definite operator, $A$, with dense domain in $\cH$ so that $\nbdom q = \mr{Dom}\, \sqrt{A}$ and
$$ q (g ,h ) = q_A (g, h) := \ip{ \sqrt{A} g }{\sqrt{A} h} _\cH; \quad \quad g,h \in \nbdom q, $$
\cite[Chapter VI, Theorem 2.1, Theorem 2.23]{Kato}.  This can be viewed as an extension of the Riesz representation lemma for bounded positive semi-definite quadratic forms. The following definition is taken from \cite[Chapter VI]{Kato}.

\begin{defn}{ (A partial order on positive semi-definite forms)} \label{porder}
If $q_1, q_2$ are densely--defined and positive semi-definite quadratic forms in a Hilbert space, $\cH$, we say $q_1 \leq q_2$ if
\bn
    \item $\nbdom q_2 \subseteq \nbdom q_1$, and
    \item $q_1 (x,x) \leq q_2 (x,x)$ for all $x \in \nbdom q_2$.
\en
If $A_1 , A_2$ are self-adjoint and positive semi-definite operators with dense domains in a Hilbert space, $\cH$, we say that $A_1 \leq A_2$ if $q_{A_1} \leq q _{A_2}$. That is, $A_1 \leq A_2$ if
\bn
    \item $\nbdom \sqrt{A_2} = \nbdom q_{A_2} \subseteq \nbdom q_{A_1} = \nbdom \sqrt{A_1}$, and 
    \item $ \| \sqrt{A} _1 x \| ^2 \leq \| \sqrt{A_2} x \| ^2$ for all $x \in \nbdom \sqrt{A_2}$.
\en
\end{defn}

\begin{thm}{ (\cite[Chapter VI, Theorem 2.21]{Kato}, \cite[Proposition 1.1]{Simon1})} \label{invorder}
Self-adjoint, densely--defined and positive semi-definite operators obey $T_1 \leq T_2$ if and only if $(t I + T_2) ^{-1} \leq (t I +T_1 ) ^{-1}$ for any $t >0$.
\end{thm}

The following lemma is a variant of \cite[Theorem 2]{DFL}:
\begin{lemma}[Unbounded Douglas factorization lemma] \label{modDFL}
Let $A, B$ be closed linear operators densely--defined in a separable Hilbert space, $\mc{H}$. If $AA^* \leq B B ^*$ (in the quadratic form sense of Definition \ref{porder}), then there is a contraction, $C$ so that $A \subseteq B C$.
\end{lemma}
\begin{proof}
This is essentially the first part of \cite[Theorem 2]{DFL}, however the partial order on positive semi-definite closed operators is defined slightly differently in \cite{DFL} than in Definition \ref{porder}. We verify that the same arguments as in \cite{DFL} work for the quadratic form partial order defined above. Assuming then that $AA^* \leq B B^*$ in the sense of positive semi-definite quadratic forms we have that $\nbdom B^* = \nbdom \sqrt{BB^*} \subseteq \nbdom \sqrt{AA^*} = \nbdom A^*$, and for any $x \in \nbdom \sqrt{BB^*}$,
$$ q_{AA^*} (x,x) = \ip{A^*x}{A^*x}_\cH \leq \ip{B^* x}{B^* x}_\cH = q_{BB^*} (x,x). $$ 
We define a linear map $D _0 : \nbran B^* \rightarrow \nbran A^*|_{\nbdom B^* }$ by:
$$ D _0 B^* x := A^* x. $$ The above quadratic form inequality implies $D _0$ is a contraction and hence extends by continuity to a contraction, $\ov{D}  _0$, from $\ran{B^*} ^{-\| \cdot \|}$ into $\ran{A^*} ^{-\| \cdot \|}$. We then extend $\ov{D}  _0$ by $0$ to a contraction, $D$, defined on all of $\cH$.  Namely, $D x = \ov{D} _0 x$ if $x \in \ran{B^*} ^{-\| \cdot \|}$ and $D x =0$ if $x$ is orthogonal to $\nbran B^*$. Let $C := D^*$ so that $D = C^*$. Since $\nbdom B^* \subseteq \nbdom A^*$, we have that $C^* B^* \subseteq A^*$ in the sense that $C^* B^* | _{\nbdom A^*} = A^*$. In particular it follows that the operator $C^* B^*$ is closeable (it has $A^*$ as a closed extension), and one can verify that its adjoint is $B C$. Since $C^* B^* \subseteq A^*$, it further follows that $A \subseteq B C$. 
\end{proof}

\section{Proof of Sufficiency}

Let $A \in \scr{L} (\cH )$ be injective, self-adjoint, multiplicity--free and not invertible.
Consider the positive semi-definite quadratic forms,
\begin{align*} q (x,y) &:= \ip{  A^{-1} x}{A ^{-1} y}; && \nbdom q = \nbran A \\
\mbox{and} \quad q_0 (x,y) &:= \ip{\inv{A} _0  x}{\inv{A} _0  y} &&\nbdom q_{0} = \nbdom \inv{A} _0 \subsetneqq \nbdom A^{-1} = \nbran A, \end{align*} where $\inv{A} _0$ is the simple symmetric operator with indices $(1,1)$ constructed in Subsection \ref{simsym}. Since $\nbdom \inv{A} _0  \subsetneqq \nbdom A ^{-1}$ by construction, we have that $q \leq q_0$ in the quadratic form sense of Definition \ref{porder}. Moreover, since $q_0 \subseteq q$, $q$ is a closed extension of $q_0$, $q_0$ is closeable and has a minimal closed extension. In fact, since $\inv{A} _0$ is closed, one can check directly that $q_0$ with form domain $\nbdom q_0 = \nbdom \inv{A} _0$ is closed. By the unbounded Riesz lemma, \cite[Chapter VI, Theorem 2.1, Theorem 2.23]{Kato}, there is a unique, self-adjoint (hence closed) and positive operator $\inv{B}$ so that $q_0 (x,y) = \ip{\inv{B} x}{\inv{B} y}$ and
$$ \nbdom \inv{B} = \nbdom \inv{A} _0. $$ (By the uniqueness of the Riesz representation of positive quadratic forms, this implies that $\inv{B} = | \inv{A} _0 | = \sqrt{\inv{A} _0 ^* \inv{A} _0}  >0$, where $\inv{A} _0  = W_0 | \inv{A} _0 |$ is the polar decomposition of the closed operator $\inv{A} _0$.) Theorem \ref{invorder} and a routine functional calculus argument imply that $B := \inv{B} ^{-1} >0$ is bounded and obeys $B^2 \leq A^2$ so that $\inv{B} = B^{-1}$. First observe that $\inv{B} >0$ is injective and self-adjoint, and so it has a self-adjoint inverse, $B>0$, with $\nbdom B := \nbran \inv{B}$. Indeed, if there is a non-zero $x \in \nbdom \inv{B}$ so that $\inv{B} x =0$, then $q \leq q_0$ implies that $A^{-1}x=0$ as well, contradicting the fact that $A^{-1}$ has a bounded inverse. Set $B := \inv{B} ^{-1}$ and write $B^{-1}$ in place of $\inv{B}$. By Theorem \ref{invorder}, since $A^{-2} \leq B^{-2}$ in the sense of Definition \ref{porder}, for any $t>0$,
$$ ( t I + B^{-2} )^{-1} \leq (tI + A^{-2} ) ^{-1}. $$ Since $A^2$ is bounded, the spectrum of $A^{-2}$ is contained in $[ \| A \| ^{-2}, + \infty )$. Taking $t =1$, spectral mapping implies that if $f(x) := (1 +x ) ^{-1}$, the spectrum of $(I + A ^{-2} ) ^{-1} = f (A^{-2} )$ is contained in $[0 , (1+ \| A \| ^{-2} ) ^{-1} ]$. The inequality $(I + B^{-2} ) ^{-1} \leq (I + A^{-2} ) ^{-1}$ then implies that the spectrum of $B ^{-1}$ is also contained in $[ \| A \| ^{-2}, + \infty )$. Indeed, if there were a $0 \leq  t < \| A \| ^{-2}$ in the spectrum of $B^{-2}$, then $(1 +t ) ^{-1} > (1+ \| A \| ^{-2}) ^{-1}$ would again, by spectral mapping, belong to the spectrum of $f(B^{-2} ) = (I + B ^{-2} ) ^ {-1}$, contradicting the fact that $f(B^{-2} ) \leq f (A^{-2})$. Since the spectrum of $B^{-2} > 0$ and hence of $B^{-1} >0$ is bounded below by a strictly positive constant, $B = (B^{-1} ) ^{-1}$ is bounded. For any $t >0$, if $f_t (x) := (t +x ) ^{-1}$, then $f_t (x)$ is continuous on $[ \eps , + \infty )$, for any fixed $\eps >0$, and $f_t$ converges uniformly on this interval to $f_0$, $f_0 (x) = x^{-1}$, as $t \downarrow 0$. The functional calculus for self-adjoint operators then implies that $f_t (B ^{-2}) = B^2 (tB^2 +I ) ^{-1}$ and $f_t (A^{-2} )$ both converge in operator--norm to $B^2$ and $A^2$ respectively. This, and the inequality $f_t (B^{-2} ) \leq f_t (A^{-2})$ for all $t>0$ imply that $B^2 \leq A^2$.

By the unbounded Douglas factorization lemma, since $A^{-2} \leq B^{-2}$ in the sense of Definition \ref{porder}, there is a contraction, $C$, so that $C^* B^{-1} \subseteq A^{-1}$, \emph{i.e.}
$$ C^* B^{-1} x  = A^{-1} x = \inv{A} _0  x, \quad \quad \forall \ x \in \nbdom \inv{A} _0 = \nbdom B^{-1} = \nbran B. $$ Since the forms $q_0$ and $q$ associated to $B^{-1}$ and $A^{-1}$ agree on $\nbdom \inv{A} _0 = \nbdom B^{-1} \subseteq \nbdom A^{-1}$, $C^* = V$ is an isometry with range
$$ \nbran V =  A ^{-1} \nbdom \inv{A} _0 = \nbran \inv{A} _0, $$ and $V B ^{-1} = \inv{A}_0$. Here, the range of $\inv{A} _0$ is closed since $\inv{A}_0$ is closed, symmetric and \emph{bounded below in norm}. Namely, since $A$ is bounded and $$ \nbdom \inv{A} _0 = \nbdom B^{-1} = \nbran B \subseteq \nbdom A^{-1} = \nbran A, $$ given any $x = Ay \in \nbran B \subseteq \nbran A$,
$$ \|  \inv{A} _0 x \|  =  \| A^{-1} Ay \| = \| y \|, $$ and $\| A y \|  \leq \| A \| \, \| y \|$ so that 
\be  \| \inv{A} _0 x \| \geq \frac{1}{\| A \|} \| x \|. \label{bbelow} \ee
This inequality, and the fact that $\inv{A} _0$ is a closed operator imply that $\nbran \inv{A} _0$ is closed. Hence,
\be V B ^{-1} = \inv{A} _0 \subseteq A ^{-1}, \quad \quad \nbran V = \nbran \inv{A} _0 = A^{-1} \nbran B, \label{DFLisom} \ee and $B ^{-1} \subseteq V^* A ^{-1}$.

\begin{lemma} \label{HankdoubleHank}
The self-adjoint operator, $A$, is Hankel with respect to $V$ and $B = V^*A = AV >0$. 
\end{lemma}
\begin{proof}
Since $B>0$ is bounded, its inverse, $B^{-1}>0$ is surjective. In particular, any $x \in \cH$ has the form $x = B ^{-1} y$.  Then,
$$ V x = V B ^{-1} y = A ^{-1} y, $$ and 
$$ A V x = y = B x. $$ This proves that $A V = B >0$, and so by taking adjoints,
$$ V^* A = A V = B > 0. $$ In particular, $A$ is a Hankel operator with respect to $V$.
\end{proof}

\begin{remark} \label{BVHankel}
Further note that $B>0$ is also $V-$Hankel since 
$$ V^* B = V^* A V = A V^2 = B V. $$ This also shows that $V^* \nbran B \subseteq \nbran B$.
\end{remark}

The conditions that $A>0$ is $V-$Hankel or that $A>0$ is $V-$Hankel and doubly--positive place certain restrictions on both $A$ and the isometry $V$. 

\begin{prop} \label{pureHank}
Suppose that $A \in \scr{L} (\cH )$ is positive, injective and Hankel with respect to an isometry $V$. Then the Wold decomposition of $V$ is $V = V_0 \oplus J$ acting on $\cH = \cH _0 \oplus \cH '$ where $V_0$ is a pure isometry and $J$ is a symmetry on $\cH' = \cH ' _+ \oplus \cH ' _-$, $J \cH ' _{\pm} = \pm \cH ' _\pm$. The subspaces $\cH _0$, $\cH ' _+$ and $\cH ' _-$ are each invariant for $A$. If $A$ is also doubly--positive with respect to $V$, \emph{i.e.} if $AV >0$, then $\cH ' _- = \{ 0 \}$ and $J=I$.
\end{prop}

\begin{proof}
Consider the Wold decomposition $V = V_0 \oplus U \simeq \bsm V_0 & \\ & U \esm$ of $V$ on $\cH = \cH _0 \oplus \cH '$, where $V_0 = V | _{\cH _0}$ is a pure isometry and $U = V |_{\cH '}$ is unitary. Then consider the corresponding block decomposition,
$$ A =: \bpm A_1 & C \\ C^* & A_2 \epm. $$ The Hankel condition, $V^* A = A V$, then implies that 
$$ \bpm V_0 ^* A _1 & V_0 ^* C \\ U^* C^* & U^* A_2 \epm = V^* A = A V = \bpm A_1 V_0 & CU \\ C^* V_0 & A_2 U \epm. $$ In particular, $U^* A_2 = A_2 U$ and $A_2$ is $U-$Hankel. Hence,
$$ U^* A_2 ^2 U = A_2 U U^* A_2 = A_2 ^2, $$ or equivalently, $U^* A_2 ^2 = A_2 ^2 U^*$. Since $A$ is positive and injective, so is $A_2$. Indeed, given any $ 0 \oplus h' \in  \cH _0 \oplus \cH ' = \cH$, $h' \neq 0$, 
$$ 0 < \ip{ 0 \oplus h '}{A  (0 \oplus h' )} _{\cH} = \ip{h'}{A_2 h'}_{\cH '}. $$ By the functional calculus, since $A_2 ^2$ commutes with $U^*$, so does $\sqrt{A_2 ^2} = A_2 >0$. That is, we have that both $U^* A_2 = A_2 U^*$ and also $U^* A_2 = A_2 U$ since $A_2$ is $U-$Hankel. In conclusion,
$$ A_2 (U-U^* ) =0, $$ and injectivity of $A_2$ then implies that $U=U^*=:J$ is a self-adjoint unitary, \emph{i.e.} a symmetry. Hence, as described in Remark \ref{symmetry}, $\cH '$ decomposes as $\cH ' = \cH ' _+ \oplus \cH ' _-$ where $J h _\pm = \pm h _\pm$ for any $h_\pm \in \cH ' _\pm$. If $A_2 = \bsm a & b \\ b^* & c \esm >0$ is the block decomposition of $A_2$ with respect to $\cH ' = \cH ' _+ \oplus \cH ' _-$, then the block decomposition of $J$ is $J = \bsm I & 0 \\ 0 & - I \esm$, and $JA_2 = A_2 J$ implies that 
$$ \bpm a & b \\ -b^* & -c \epm = JA_2 = A_2 J = \bpm a & -b \\ b^* & - c \epm. $$ It follows that $b = - b$ so that $b\equiv 0$. Now, $V= \bsm V_0 & 0 \\ 0  & J \esm$ on $\cH = \cH _0 \oplus \cH '$, where $J^2 =I$ so that
\ba \bpm V_0 ^{*2n} A_1 & V_0 ^{*2n} C \\ C^* & A_2 \epm & = & \bpm V_0 ^{*2n} & 0 \\ 0  & I \epm \bpm A_1 & C \\ C^* & A_2 \epm \nn \\
& = & V ^{*2n} A  \nn \\
& = & A V^{2n} \nn \\
& = & \bpm A_1 & C \\ C^* & A_2 \epm \bpm V_0 ^{2n} &  0 \\ 0 & I \epm \nn \\
& = & \bpm A_1 V_0 ^{2n} & C \\ C^* V_0 ^{2n} & A_2 \epm. \nn \ea 
In particular, $V_0 ^{*2n} C = C$ which implies that $C \equiv 0$ since $V_0$ is a pure isometry and $V_0 ^{*n} \rightarrow 0$ in the strong operator topology. This proves that $A >0$ is block diagonal with respect to the Wold decomposition of $V = V_0 \oplus J$ and $\cH = \cH _0 \oplus \cH '$. Since both $C \equiv 0$ and $b \equiv 0$, $A$ is block diagonal with respect to $\cH = \cH _0 \oplus \cH ' _+ \oplus \cH ' _-$, $A = \bsm A_1 & 0 & 0 \\ 0 & a & 0 \\ 0 & 0 & c \esm$, and each of $\cH _0$, $\cH ' _+$ and $\cH ' _-$ are $A-$invariant. If $A$ is doubly--positive and Hankel with respect to $V$, then both $A_2 = \bsm a & 0 \\ 0 & c \esm >0$ and $A_2J = \bsm a & 0 \\ 0 & -c \esm >0$. Since $A_2$ is injective, this implies that $\cH ' _- = \{ 0 \}$ so that $\cH ' = \cH ' _+$ and $J =I$. 
\end{proof}

Theorem \ref{Hdpos} is a straightforward consequence of the previous results:

\begin{proof}{ (of Theorem \ref{Hdpos})}
If $A = J|A|$ is $V-$Hankel where $J = J^* = J^{-1}$ is a symmetry and $W:= JV$ then,
\ba W^* |A| & = & V^* J ^* |A| = V^* J |A|  \nn \\
& = & V^* A = A V  \nn \\
& = & |A| J^* V = |A| JV = |A| W. \nn \ea Hence, $|A| >0$ is $W-$Hankel. Moreover, by assumption, $V^* A  = AV >0$ so that 
$W^* |A | = V^* J |A| = V^* A >0$. It follows that $|A|$ is $W-$Hankel and doubly--positive. It is easily checked that $W$ has pure part of multiplicity $N$ if $V$ has multiplicity $N$. Finally, if the Wold decomposition of $W$ is $W=W_0 \oplus U$ on $\cH = \cH _0 \oplus \cH '$, where $W_0$ is pure and $U$ is unitary, the previous proposition implies that both $\cH _0, \cH '$ are $|A|-$invariant and that either $U = I _{\cH ' }$ or $\cH ' = \{ 0 \}$.
\end{proof}

The following lemma is essentially a special case of a construction known to M.G. Kre\u{\i}n, \cite[Section 1.2]{Krein}:

\begin{lemma} \label{Dspace}
Let $T_0$ be a closed, symmetric operator with dense domain, $\nbdom T_0$, in a Hilbert space, $\cH$. Further assume that $T_0$ has a self-adjoint extension, $T$, with domain in $\cH$ such that $T$ has a bounded inverse, $T^{-1} \in \scr{L} (\cH )$. If $X := T (T -iI ) ^{-1} = (I - i T^{-1} ) ^{-1}$, then $X$ implements a bijection from $\nbran T_0 ^\perp$ onto $\ran{T_0 +i I} ^\perp$. In particular,
$$n = \nbdim \nbran T_0 ^\perp = \nbdim \ran{T_0 +i I} ^\perp, $$ and $T_0$ has equal deficiency indices $(n,n)$. 
\end{lemma}

\begin{proof}
Since $T_0$ has a self-adjoint extension with dense domain in $\cH$, namely $T$, $T_0$ must have equal deficiency indices $(n,n)$, $n \in \N \cup \{ 0 \} \cup \{ + \infty \}$ \cite[Sections 78--81]{Glazman}, \cite[Chapter X.1]{RnS2} so that $\nbdim \ran{T_0 +i I} ^\perp =n$. 
Note that $X$ is bounded, normal and invertible with inverse $Y=I -iT^{-1}$.
Given any $g \in \nbdom T_0$ and any $h \perp \nbran T_0 $, \emph{i.e.} any $h$ orthogonal to $\nbran T_0$, calculate
\ba \ip{(T_0 + iI)g}{Xh} & = & \ip{(T+iI)g}{T (T -i I)^{-1} h} \nn \\
& = & \ip{Tg}{h} \nn\\
& = & \ip{T_0 g}{h} =0. \nn \ea 
Since $X$ is injective (it is invertible), this proves that $X$ is an injective map from $\nbran T_0 ^\perp$ into $\ran{T_0 + iI} ^\perp$. Conversely, given any $g \in \nbdom T_0$ and any $h \perp \ran{T_0 +iI}$,
\ba \ip{T_0g}{Yh} & = & \ip{Tg}{(I-iT^{-1})h} \nn \\
& = & \ip{(T+iI)g}{h} \nn \\
& = & \ip{(T_0 +iI )g}{h} =0. \nn \ea 
Hence, $Y=X^{-1}$ is an injective map from $\ran{T_0 +iI} ^\perp$ into $\nbran T_0 ^\perp$. It follows that $X : \nbran T_0 ^\perp \rightarrow \ran{T_0 +iI} ^\perp$ is also surjective, hence bijective: If $f \in \ran{T_0 +iI} ^\perp$, then $h:=Yf \in \nbran T_0 ^\perp$, and then $Xh = XY f  =f \in X \, \nbran T_0 ^\perp$. 
\end{proof}

We are now sufficiently prepared to complete the proof of our main result, Theorem \ref{main}:

\begin{proof}{ (Sufficiency in Theorem \ref{main})}
Recall that we assume, without any loss in generality, that $A=A^* \in \scr{L} (\cH )$ is multiplicity--free. We have constructed an isometry, $V$, so that $A$ obeys the Hankel condition with respect to $V$, where $A=A^*$ is injective but not invertible, and we have constructed a closed, densely--defined and simple symmetric operator $\inv{A} _0 \subseteq A^{-1}$ with deficiency indices $(1,1)$, domain $\nbdom \inv{A} _0 = \nbran AV$ and $\nbran \inv{A} _0 = \nbran V$. We further observed in Remark \ref{BVHankel} above, that $B = V^* A = AV$ is positive and $V-$Hankel. By Proposition \ref{pureHank}, since $B$ is positive, injective and $V-$Hankel, the Wold decomposition of $V$ is $V = V_0 \oplus J$ on $\cH = \cH _0 \oplus \cH '$, where $V_0$ is pure, $J$ is a symmetry and $B = \bsm B_1 & 0 \\ 0 & B_2 \esm$ is block diagonal with respect to this decomposition. Hence, if $A = \bsm A_1 & C \\ C^* & A_2 \esm$ is the corresponding block decomposition of $A$, we have 
$$ AV = \bpm A_1 & C \\ C^* & A_2 \epm \bpm V_0 & 0 \\ 0 & J \epm = \bpm A_1 V_0 & C J  \\ C^* V_0 & A_2 J \epm = \bpm B_1 & 0 \\ 0 & B_2 \epm = B. $$
In particular, $CJ =0$, and it follows that $C \equiv 0$ since $J$ is unitary. This proves that $A = \bsm A_1 & 0 \\ 0 & A_2 \esm$ is also block diagonal with respect to the Wold decomposition of $V$. It remains to show that $V$ is pure and unitarily equivalent to one copy of the shift. Since $V = V_0 \oplus J$,
$$ \nbran V = \nbran V_0 \oplus \cH ' = \nbran \inv{A} _0. $$ In particular, given any $x_0 \in \nbran V_0$ and any $x' \in \cH '$, $x = x_0 \oplus x' \in \nbran V$ and both  $x_0 \oplus 0$ and $0 \oplus x'$ also belong to $\nbran V = \nbran \inv{A}_0$. Hence, given any $x' \in \cH'$, there is an $h = h _0 \oplus h' \in \nbdom \inv{A}_0$ so that $\inv{A}_0 h = 0 \oplus x'$. Since $\inv{A} _0 \subseteq A^{-1}$, it follows that 
\ba 0 \oplus x' & = & \inv{A} _0 (h_0 \oplus h' ) \nn \\
& = &   A ^{-1} (h_0 \oplus h' ). \nn \ea 
Applying $A$ to both sides of this equation gives $$ A (0 \oplus x') = 0 \oplus A_2 x' = h_0 \oplus h', $$ so that $h_0 =0$, $h' = A_2 x'$ and $A ( 0 \oplus x') = 0 \oplus A_2 x' = 0 \oplus h' \in \nbdom \inv{A} _0$. In conclusion,
$$ \inv{A} _0 (0 \oplus A_2 x') = A^{-1} ( 0 \oplus A_2 x') = 0 \oplus x', $$ 
so that 
$$ \inv{A} _0 | _{\{ 0 \} \oplus \nbran A_2} = A^{-1} | _{\{ 0 \} \oplus \nbran A_2} = 0 \oplus A_2 ^{-1} | _{\{ 0 \} \oplus \nbran A_2}, $$ is self-adjoint, contradicting the simplicity of $\inv{A} _0$. This proves that $V=V_0$ must be pure. Since $\inv{A} _0$ has deficiency indices $(1,1)$, we have that
$$ \nbdim \nbran (\inv{A} _0 \pm i I) ^\perp =1, $$ and Lemma \ref{Dspace} then implies that $\nbdim \nbran \inv{A} _0  ^\perp = \mr{dim} \, \nbran V ^\perp =1$.  Hence $V \simeq S$ is unitarily equivalent to the unilateral shift.
\end{proof}

We now prove our second main result, Theorem \ref{main2}:

\begin{proof}{ (of Theorem \ref{main2})}
Suppose that $A \in \scr{L} (\cH )$ satisfies the assumptions of the theorem statement. Namely, $A$ is self-adjoint, injective and not invertible so that $A^{-1}$ with domain $\nbdom A^{-1} = \nbran A$ is densely--defined and self-adjoint. Suppose that $\inv{A} _0 \subseteq A^{-1}$ is a closed, symmetric and non-self--adjoint restriction of $A^{-1}$ to some dense domain in $\cH$ with deficiency indices $(N,N)$ where $N \in \N \cup \{ + \infty \}$. (In fact, the assumption that $\inv{A} _0$ has $A^{-1}$ as a self-adjoint extension implies that $\inv{A} _0$ must have equal deficiency indices \cite[Sections 78--81]{Glazman}, \cite[Chapter X.1]{RnS2}, \cite{vN}.) As before we can define two closed and positive semi-definite quadratic forms, $q, q_0$ with form domains $\nbdom q = \nbdom A^{-1} = \nbran A$ and $\nbdom q_0 = \nbdom \inv{A} _0 \subseteq \nbdom q$ by
$$ q(x,y) := \ip{A^{-1} x}{A^{-1}y}; \quad x,y \in \nbdom q, \quad \mbox{and} \quad q_0 (x,y) = \ip{A^{-1} x}{A^{-1} y}; \quad x,y \in \nbdom q _0.$$
By construction, $q_0 \subseteq q$ and $q \leq q_0$. Also as before, by the unbounded Riesz lemma, there is a self-adjoint and positive linear operator $\inv{B}>0$ so that 
$$ q_0 (x,y) = \ip{\inv{B} x}{\inv{B} y}, \quad \quad \nbdom \inv{B} = \nbdom q_0, $$ and $\inv{B}$ has a bounded inverse, which we denote by $B = \inv{B} ^{-1} >0$. (Hence $\inv{B} = B^{-1}$ and we will write $B^{-1}$ in place of $\inv{B}$.) Again, by Douglas factorization there is an isometry, $V$, with $\nbran V = \nbran \inv{A} _0$, so that 
$V B^{-1} = \inv{A} _0$, and the argument of Lemma \ref{HankdoubleHank} implies that $B = V ^* A = A V$ so that $A$ is $V-$Hankel. Further note that $q_0 = q_{\inv{A} _0 ^* \inv{A} _0} = q_{B^{-2}}$ implies that $B^{-1} = |\inv{A} _0 |$ \cite[Chapter VI, Theorem 2.1, Theorem 2.23]{Kato}. This, and the equality $VB^{-1} = \inv{A} _0$ imply that if the (unique) polar decomposition of the closed operator $\inv{A} _0$ is $\inv{A} _0 = V' | \inv{A} _0 | = V' B^{-1}$, then $V' =V$. By Lemma \ref{Dspace}, 
$$ \nbdim \nbran V ^\perp = \nbdim \ran{\inv{A} _0 + iI} ^\perp = N, $$ so that the pure part of $V$ has multiplicity $N \neq 0$. Arguments similar to those in the proof of Theorem \ref{main} show that if we further assume that $\inv{A} _0$ is simple, then this assumption and the facts that both $B= V^*A >0$ and $A$ are $V-$Hankel, imply that $V$ is pure.

Conversely, suppose that $A$ is self-adjoint, injective and not invertible, and that $A$ is Hankel with respect to some isometry $V$ with pure part, $V_0$, of multiplicity $N>0$, where $B := V^*A = AV$. Note that $B$ need not be positive semi-definite here. Define the positive semi-definite quadratic forms 
$$ q_0 (x,y) := \ip{B^{-1} x}{B^{-1}y}_{\cH}, \quad \quad x,y \in \nbran B \subseteq \nbran A, $$ and 
$$ q(x,y) := \ip{ A^{-1} x}{A^{-1} y}_\cH, \quad \quad x,y \in \nbran A. $$
We claim that $q_0 \subseteq q$ and $q \leq q_0$. Indeed, $\nbdom q_0 = \nbran B \subseteq \nbran A = \nbdom q$, and if $x =By =AV y \in \nbdom q_0$, then $A^{-1}x = Vy \in \nbran V$. Hence, for any $x = AV y$ and $x' = AV y'$ in $\nbdom q_0$,
\be q_0 (x,x') = \ip{y}{y'}_{\cH} = \ip{Vy}{Vy'}_{\cH} = \ip{A^{-1} AVy}{A^{-1}AVy'}_\cH = q (x,x'). \label{qeq} \ee 
Defining $\inv{A} _0$ as the restriction of $A^{-1}$ to $\nbdom B ^{-1} = \nbran B$, we will show that $\inv{A} _0$ is closed, symmetric and has equal, non-zero deficiency indices. To show that $\inv{A} _0$, with domain $\nbran B$ is closed, suppose that $x_n \in \nbran B$, $x_n \rightarrow x$, and that $\inv{A} _0 x_n =A^{-1} x_n \rightarrow y$. Then $x_n = B h_n = AV h_n \rightarrow x$ and $A^{-1} x_n = Vh_n \rightarrow y$. Since $V$ is an isometry, it has closed range and $y = Vh$ for some $h \in \cH$. Since $A$ is bounded, $x_n = AVh_n \rightarrow Ay = AVh = Bh =x \in \nbran B$ and $y = Vh = A^{-1} AVh = A^{-1}x = \inv{A} _0 x$. This proves that $\inv{A} _0$ is closed on $\nbdom \inv{A} _0 = \nbran B$. (In general any densely--defined symmetric linear operator is closeable, \cite[Theorem 1, Section 46]{Glazman}.) Since $\inv{A} _0 \subseteq A^{-1}$ has $A^{-1}$ as a self-adjoint extension, it must have equal deficiency indices $(n,n)$, $n \in \N \cup \{ 0 \} \cup \{ + \infty \}$ \cite[Sections 78--81]{Glazman}, \cite[Chapter X.1]{RnS2}, \cite{vN}.  By construction, $\nbran \inv{A} _0 = A^{-1} \nbran B = \nbran V \subsetneqq \cH$ is a closed and proper subspace of $\cH$.  Since we assume that $N = \nbdim \nbran V ^\perp$ and we have shown that $\nbran \inv{A} _0 ^\perp = \nbran V ^\perp$, Lemma \ref{Dspace} implies that $n=N>0$. Since, $\inv{A} _0$ is closed, we further obtain that $\inv{A} _0 = V B^{-1}$ where $B = V^*A = AV$.

It remains to show that purity of $V$ and the additional assumption that $B >0$ imply simplicity of $\inv{A} _0$. Suppose that $V$ is pure and $\inv{A} _0$ is not simple. Then there is a (non-closed) linear subspace, $\scr{D} \subseteq \nbdom \inv{A} _0 = \nbran B$, with closure $\cJ := \ov{\scr{D}}$, so that $\cJ$ is a closed, non-trivial subspace of $\cH$, and $\inv{A_1} := \inv{A} _0 | _{\scr{D}}$ is densely--defined in $\cJ$ and self-adjoint. Since, as shown in Equation (\ref{bbelow}), $\inv{A} _0$ is bounded below in norm, so is $\inv{A_1}$ so that $\inv{A_1}$ is surjective and has a bounded, self-adjoint inverse, $A_1 :=  \inv{A_1}  ^{-1}$ and $\inv{A_1} = A_1^{-1}$. We claim that $\cJ$ is $A-$invariant and that $A | _{\cJ} =A_1$. Indeed, first note that $A_1 ^{-1} = \inv{A} _0 | _{\scr{D}} = A^{-1}| _{\scr{D}}$ by definition. Moreover, as observed above, $A_1 ^{-1}$ is surjective on $\cJ$. Hence, given any $x \in \cJ$, there is a $y \in \scr{D} \subseteq \cJ$ so that $x = A_1 ^{-1} y = A^{-1}y$ and then $Ax = y = A_1 x \in \cJ$. This proves that $A \cJ \subseteq \cJ$ so that $\cJ$ is $A-$invariant. Further note that $\scr{D} = \cJ \cap \nbran B = \nbdom \inv{A} _0$. Certainly $\scr{D} \subseteq \cJ \cap \nbdom \inv{A} _0 =  \cJ \cap \nbran B$, and if $\scr{D} \subsetneqq \cJ \cap \nbran B =: \scr{D} '$, then $\inv{A} _0 | _{\scr{D} '}$ would be a symmetric (hence closeable) extension of the self-adjoint $A_1 ^{-1} = \inv{A} _0 | _{\scr{D}}$ acting on a dense domain in $\cJ$, and no such non-trivial extensions exist \cite[Sections 78--81]{Glazman}, \cite[Chapter X.1]{RnS2}, \cite{vN}. Hence, if $x = By = AV y \in \scr{D} = \nbran B \cap \cJ$, then $A_1 ^{-1} x = A^{-1} x = Vy \in \nbran V$. Since $A_1 ^{-1}$ maps $\scr{D} \subseteq \cJ$ onto $\cJ$, this implies $\cJ \subseteq \nbran V$. This implies, in turn, that $\cJ$ is $B^2-$invariant,
$$ B^2 \cJ = AV V^* A \cJ \subseteq A VV^* \cJ = A \cJ \subseteq \cJ, $$  and since we assume that $B>0$, $\cJ$ is also $B-$invariant so that 
$$ \cJ \supseteq B \cJ = V^* A \cJ = V^* A_1 \cJ. $$ This shows, finally, that $V^*$ maps $\nbran A_1 \subseteq \cJ$, which is dense in $\cJ$, into $\cJ$ and it follows that $\cJ$ is also $V^*-$invariant. However, $\cJ \subseteq \nbran V$ so that given any $x =Vy \in \cJ$,
$$ \| V^* x \| = \| V^*V y\| = \| y \| = \| V y \| = \| x \|, $$ and $V^* | _{\cJ}$ is isometric contradicting the purity of $V$. Namely, $V^{*k}$ converges to $0$ in the strong operator topology, so that $V^*$ cannot have any isometric restriction to an invariant subspace. This contradiction proves that $\inv{A} _0$ must be simple.
\end{proof}


\Addresses

\end{document}